\documentclass[12pt,reqno]{amsart}
\usepackage{amssymb,amsbsy}
\usepackage{amscd}
\usepackage{enumerate}
\usepackage{setspace}
\usepackage{tikz}
\usepackage{bm}
\usepackage{cite}
\usepackage{graphicx}
\usepackage{hyperref}
\usepackage{nicematrix}
\usepackage{colortbl}
\usepackage{mathtools}

\allowdisplaybreaks

\usepackage{kpfonts}
\usepackage{stackengine}
\usepackage{calc}
\newlength\shlength
\newcommand\xshlongvec[2][0]{\setlength\shlength{#1pt}%
  \stackengine{-5.6pt}{$#2$}{\smash{$\kern\shlength%
    \stackengine{7.55pt}{$\mathchar"017E$}%
      {\rule{\widthof{$#2$}}{.57pt}\kern.4pt}{O}{r}{F}{F}{L}\kern-\shlength$}}%
      {O}{c}{F}{T}{S}}
  
\usepackage[letterpaper]{geometry}
\theoremstyle{definition}
\newtheorem{theorem}{Theorem}[section]
\newtheorem{lemma}[theorem]{Lemma}
\newtheorem{proposition}[theorem]{Proposition}
\newtheorem{definition}[theorem]{Definition}
\newtheorem{corollary}[theorem]{Corollary}

\newcommand{\ovl}[1]{\overline{#1}}	%El conjugado
\newcommand{\parn}[1]{\left( #1 \right)}%Parentesis
\newcommand{\conv}[1]{\mathrm{Conv}\parn{#1}}	%Envolvente convexa
\newcommand{\Vr}{\overrightarrow{V_r}}

\newcommand{\Z}{{\mathbb{Z}}}
\newcommand{\R}{{\mathbb{R}}}
\newcommand{\C}{{\mathbb{C}}}
\newcommand{\N}{{\mathbb{N}}}

\newcommand\eit{e^{i\theta}} 
\newcommand\neit{e^{-i\theta}} 

\renewcommand{\Re}{\ensuremath{\mathrm{Re}}}
\renewcommand{\Im}{\ensuremath{\mathrm{Im}}}

\newcommand{\T}{\mathsf{T}}

\begin{document}

\title{The numerical range of periodic banded Toeplitz operators}
 
\author[B.A. Itz\'a-Ortiz, R.A. Mart\'inez-Avenda\~no, and H. Nakazato]{Benjam\'in A.~Itz\'a-Ortiz, Rub\'en A.~Mart\'inez-Avenda\~no and Hiroshi Nakazato}


\address{Universidad Aut\'onoma del Estado de Hidalgo}
\email{itza@uaeh.edu.mx}
\address{Instituto Tecnol\'ogico Aut\'onomo de M\'exico}
\email{ruben.martinez.avendano@gmail.com}
\address{Hirosaki University}
\email{nakahr@hirosaki-u.ac.jp}

\thanks{The second author's research is partially supported by the Asociaci\'on Mexicana de Cultura A.C}


\maketitle 

 \begin{abstract}
     We prove that the closure of the numerical range of a $(n+1)$-periodic and $(2m+1)$-banded Toeplitz operator can be expressed as the closure of the convex hull of the uncountable union of numerical ranges of certain symbol matrices.  In contrast to the periodic $3$-banded (or tridiagonal) case, we show an example of a $2$-periodic and $5$-banded Toeplitz operator such that the closure of its numerical range is not equal to the numerical range of a single finite matrix.
 \end{abstract}

\section{Introduction}
Tridiagonal operators have long been a subject of interest not only for their remarkable mathematical properties   but more recently for their applications to mathematical physics. In \cite{F99} a tridiagonal operator known as the ``hopping sign model'' was introduced and subsequently studied by many researchers 
\cite{ChN2011,Ch-W2011,Hagger2015,ChS,IO-MA-N2022}.
By following the work in \cite{Bebianoetal} (see also \cite{BS}), the first two authors showed in \cite{IM} that the numerical range of the $(n+1)$-periodic tridiagonal Toeplitz operator on $\ell^2(\N_0)$,  given by
\[
\begin{pmatrix}
b_0 & c_0 & & & & & & & &\\
a_1 & b_1 & c_1 & & & & & & & \\
    & a_{2} & b_2 & c_2 & & & & & & \\
    &       & \ddots & \ddots  & \ddots & & & & & \\
    & & & a_{n} & b_n & c_n & & & &\\
    & & & & a_{0} & b_0 & c_0 & & &\\
    & & & & & \ddots & \ddots & \ddots & &\\
    & & & & & &a_{n-1} & b_{n-1}& c_{n-1} & \\
    & & & & & & &  a_n& b_n  & c_n &\\
    & & & & & & & & \ddots &  \ddots&\ddots
\end{pmatrix},
\]
is equal to the closure of the convex hull of the uncountable union of numerical ranges of certain symbol matrices. In this paper, we extend this result to periodic banded Toepliz operators of arbitrary band. While the results in \cite{Bebianoetal} already include the case for banded Toeplitz operators of arbitrary band, they are constrained to biperiodic Toeplitz operators. Hence, our results may also be regarded as a generalization of those in \cite{Bebianoetal}. Later, in \cite{IO-MA-N2022}, the authors were able to show that the closure of the numerical range of a $(n+1)$-periodic tridiagonal operator is equal to the numerical range of a single $2(n+1)\times2(n+1)$ complex matrix. It is therefore natural to ask if a similar result would also apply for general periodic banded operators. We answer this question in the negative by exhibiting a biperiodic $5$-banded Toeplitz operator for which the closure of its numerical range cannot be the numerical range of a single complex matrix.

Also, as a consequence of the above representation of the numerical range of a periodic banded Toeplitz operator as the closure of the convex hull of the union of numerical ranges of finite matrices, we obtain the result that there exists a real binary polynomial $g(x,y)$ such that the boundary of the numerical range of a periodic banded Toeplitz operator is a subset of the set of zeroes of $g(x,y)$.

We divide this work in three sections. In Section~2 we establish the definitions needed and state and prove the main result of this paper. In Section~3 we prove a consequence of the main result; namely the existence of a real binary form such that the boundary of the numerical range of a periodic banded Toeplitz operator is contained in the zero set of such form. Finally, in Section~4 we give an example of a 2-periodic and 5-banded Toeplitz operator for which the closure of the numerical range is not the numerical range of a single matrix.

As usual, we denote by $\N_0$, $\N$, $\Z$, $\R$ and $\C$, the sets of nonnegative integers, the set of positive integers, the set of integers, the set of real numbers, and the set of complex numbers, respectively.

\section{The numerical range of a periodic banded Toeplitz operator}

For $m \in \N_0$, we consider $2m+1$ complex-valued bi-infinite sequences denoted by \[
a^{(r)}=\left(a_j^{(r)}\right)_{j \in \Z},
\]
for $r=-m, \dots, 0, \dots, m$. For $n \in \N_0$, we assume that each sequence is $(n+1)$-periodic; i.e., $a_{j}^{(r)}=a_{j+n+1}^{(r)}$ for each $r=-m, \dots, m$ and each $j \in \Z$. If $|r|>m$, then we define $a_j^{(r)}=0$ for all $j \in \Z$.

We define the operator $T : \ell^2(\Z) \to \ell^2(\Z)$ by the bi-infinite matrix with $(j,k)$ entry the number $a_{j}^{(j-k)}$; i.e., by the bi-infinite matrix 
\[
T=\left( a_{j}^{(j-k)} \right)_{j, k \in \Z}.
\]
In other words, the operator $T$ is $(2m+1)$-banded and the sequence $a^{(r)}$ appears in the $r$-th diagonal, with $r=-m, \dots, 0, \dots, m$. The operator $T$ is a $(n+1)$-periodic  and $(2m+1)$-banded Toeplitz operator on $\ell^2(\Z)$. 

We define the operator $T_+: \ell^2(\N_0) \to \ell^2(\N_0)$ as $T_+=P T_{| \ell^2(\N_0)}$, where $P$ is the orthogonal projection of $\ell^2(\Z)$ onto $\ell^2(\N_0)$, thought of as a subspace of $\ell^2(\Z)$. The matrix of the operator $T_+$ is given by
\[
T_+=\left( a_{j}^{(j-k)} \right)_{j, k \in \N_0},
\]
%Observe that with the notation defined above, we have that 
that is, the $(j,k)$ entry of the matrix of $T_+$ is
\begin{equation}\label{eq:def_a}
\langle T_+ e_k, e_j \rangle = a_{j}^{(j-k)}
\end{equation}
for all $j, k \in \N_0$. In expanded form,
\[
T_+= 
\begin{pmatrix}
  a_0^{(0)} & a_0^{(1)} & \dots & a_0^{(m)}& 0 & \cdots \\[5pt]
 a_1^{(-1)} & a_1^{(0)} & a_1^{(1)} & \cdots & a_1^{(m)} & 0 & \cdots &  \\[5pt]
 \vdots & & \ddots & & & \ddots \\[5pt]
 a_m^{(-m)} & \cdots & a_m^{(-1)}& a_m^{(0)}& a_m^{(1)}& \cdots &a_m^{(m)} &\cdots \\[5pt]
 0 & a_{m+1}^{(-m)}& \cdots & a_{m+1}^{(-1)}& a_{m+1}^{(0)}& a_{m+1}^{(1)} &\cdots & \cdots \\[5pt]
  \vdots & 0 & a_{m+2}^{(-m)}& \cdots & a_{m+2}^{(-1)}& a_{m+2}^{(0)}& a_{m+2}^{(1)} &\cdots \\[5pt]
   \vdots  & \vdots & 0 & a_{m+3}^{(-m)}& \cdots & a_{m+3}^{(-1)}& a_{m+3}^{(0)}&\cdots  \\[5pt]
 \vdots & & & \ddots &\ddots & & & \ddots 
\end{pmatrix}.
\]
We also will refer to $T_+$ as an $(n+1)$-periodic and $(2m+1)$-banded Toeplitz operator, acting on $\ell^2(\N_0)$.

The operators $T$ and $T_+$ defined by the matrices above are bounded, since the sum of the moduli of each row and each column is uniformly bounded (see, e.g., \cite[Example 2.3]{Kato})

 As in \cite{IM}, we introduce the symbol $\Phi(\theta)$ of the $(n+1)$-periodic and $(2m+1)$-banded Toeplitz operators $T$ and $T_+$. We shall show that the matrix $\Phi(\theta)$ with trigonometric polynomial entries in $\theta$ plays a crucial role in expressing the closure of  the numerical range $W(T_+)$ of the operator $T_+$.

\begin{definition}\label{def:symbol}
We define the symbol $\Phi$ of a $(n+1)$-periodic and $(2m+1)$-banded Toeplitz operator $T$ on $\ell^2(\Z)$ as the function $\Phi:[0, 2 \pi) \to M_{n+1}(\C)$ given as follows. For each $j,k=0, 1, 2, \dots, n$, the $(j,k)$ entry of $\Phi(\theta)$ is
\[
\sum_{u\in \Z} e^{i u \theta} \langle T e_{k+u(n+1)}, e_j \rangle. 
\]
We define the symbol $\Phi$ of a $(n+1)$-periodic and $(2m+1)$-banded Toeplitz operator $T_+$ on $\ell^2(\N_0)$ as the symbol of $T$.
\end{definition}

{\bf Remark.} We should point out that this definition of the symbol does not coincide with the symbol $T_\theta$ defined in \cite{Bebianoetal} (cf. \cite{BS}) for the case $n+1=2$, but it is simple to see that both matrices, in this case, are unitarily equivalent: $\Phi(\theta)={\rm diag}(1, e^{-i \theta}) T_{\theta} \ {\rm diag}(1, e^{i \theta})$. On the other hand, our definition of symbol, when $m=1$, coincides with that of \cite{IM}.

\vskip0.2cm

In \cite{IM}, one can see the explicit form for the symbol when $m=1$. If $m=2$, the explicit form of the symbols $\Phi(\theta)$ will be shown next. For this, consider the $(n+1)$-periodic and $5$-banded Toeplitz operator $T$ and we set $p_j:=a_j^{(-2)}$, $q_j:=a_j^{(-1)}$, $r_j=a_j^{(0)}$, $s_j:=a_j^{(1)}$, and $t_j:=a_j^{(2)}$, with $j=0,1, 2, \ldots, n$.

If $n+1=2$, then the symbol of $T$ is the $2\times 2$ matrix 
\begin{equation*}
    \Phi(\theta)=
        \begin{pmatrix}
        p_0 e^{-i \theta} +r_0 +t_0 e^{i \theta}  & q_0 e^{-i \theta}+s_0  \\ 
        q_1 +s_1 e^{i \theta}  & p_1 e^{-i \theta} +r_1 +t_1 e^{i \theta}
        \end{pmatrix}.
\end{equation*}
If $n+1=3$, then the  symbol of $T$ is the $3\times 3$ matrix 
\begin{equation*}
    \Phi(\theta)=
        \begin{pmatrix}
          r_0  & p_0 e^{-i \theta} +s_0   & q_0 e^{-i \theta} +t_0   \\
         q_1  +t_1  e^{i \theta}  & r_1 & p_1 e^{-i \theta} +s_1   \\
        s_2 e^{i \theta} +p_2 & q_2+t_2 e^{i \theta} & r_2 
     \end{pmatrix}.
 \end{equation*} 
If $n+1=4$, then the symbol of $T$ is the $4 \times 4$ matrix 
\begin{equation*}
    \Phi(\theta)=
        \begin{pmatrix}
          r_0  & s_0  & p_0 e^{-i \theta} + t_0  & q_0 e^{-i \theta}   \\
          q_1  &  r_1 & s_1 & a_1 e^{-i \theta}  +e_1      \\
          p_2 +t_2 e^{i \theta} & q_2 & r_2 & s_2 \\
          s_3 e^{i\theta}    &  p_3 +t_3 e^{i\theta} & q_3 & r_3  
    \end{pmatrix}.
    \end{equation*} 
If $n+1\geq 5$, then the symbol of $T$ is the $(n+1)\times(n+1)$ matrix
\[
\Phi(\theta)= \begin{pmatrix}
 r_0 & s_0 & t_0 & & & p_0 e^{-i \theta} & q_0 e^{-i\theta}\\
q_1 & r_1 & s_1 & t_1 & & & p_1 e^{-i \theta} \\
p_2 & q_{2} & r_2 & s_2 & t_2 & & \\
    & \ddots & \ddots &\ddots & \ddots & \ddots   & \\
& & p_{n-2} & q_{n-2} & r_{n-2} & s_{n-2} & t_{n-2}\\
t_{n-1} e^{i \theta} & & & p_{n-1} & q_{n-1} & r_{n-1} & s_{n-1} \\
s_n e^{i\theta} & t_n e^{i\theta} & & & p_n & q_n & r_n
\end{pmatrix}.
\]

Observe that, for arbitrary $n$ and $m$ the sum defining the symbols is finite: indeed, for each $(j,k)$ we are summing over the $(j,k+u(n+1))$ entries of $T$ as $u$ ranges over the integers; that is, we are looking at the $j$-th row of $T$ and there are only finitely many nonzero entries in this row for if $|j-(k+u(n+1)|>m$, then $\langle T e_{k+u(n+1)}, e_j \rangle=0$, since $T$ is $(2m+1)$-banded.

Furthermore, notice that for $|j-(k+u(n+1)|\leq m$ to hold we need $j-(k+u(n+1))$ to be in the set $\{-m,-(m-1), \dots, -1, 0, 1, \dots (m-1), m \}$; since, as $u$ ranges over $\Z$, the different values of $j-(k+u(n+1))$ vary by multiples of $n+1$, there are at most $\left\lceil \frac{2m+1}{n+1}\right\rceil$ values of $u$ for which $-m \leq j-(k+u(n+1)) \leq m$. 

In the case $n+1 \geq 2m+1$; i.e., the period is greater or equal than the length of the band, then there is at most one integer value of $u$ for which $|j-(k+u(n+1)|\leq m$ and hence we have that $\Phi(\theta)$ equals
\begin{equation}\label{eq:matrix}
\begin{pmatrix}
  a_0^{(0)} & a_0^{(1)} & \cdots  & \cdots & a_0^{(m)} & 0 & \cdots & 0 & a_0^{(-m)} \neit & \cdots & a_0^{(-2)} \neit & a_0^{(-1)} \neit \\[5pt]
  a_1^{(-1)} & a_1^{(0)} & \ddots & &  \cdots & a_1^{(m)} & 0 & \cdots & 0 & \ddots &  & a_1^{(-2)} \neit \\[5pt]
 \vdots & \ddots & \ddots & & & & & & & & \ddots & \vdots \\[5pt]
\vdots & &  & \ddots & & & & & & & 0 & a_{m-1}^{(-m)} \neit \\[5pt]
 a_m^{(-m)} & & & & a_m^{(0)}& & & & & & & 0\\[5pt]
 0 & a_{m+1}^{(-m)} & & & & a_{m+1}^{(0)}& & & & & & \vdots \\[5pt] 
 \vdots & 0 & & & & & \ddots \\[5pt]
 \vdots & & & & & & &  \ddots & & & a_{n-m-1}^{(m)}& 0\\[5pt]
 0 & \vdots & & & & & &  &  & & & a_{n-m}^{(m)}\\[5pt]
 a_{n-m+1}^{(m)} \eit & 0 \\[5pt]
 \vdots & \ddots & & & & & & & &  \ddots\\[5pt]
 a_{n-1}^{(2)} \eit & &  & \ddots & 0 & \cdots & 0 &  a_{n-1}^{(-m)}& & & a_{n-1}^{(0)} & a_{n-1}^{(1)}\\[5pt]
  a_n^{(1)} \eit & a_n^{(2)} \eit & \cdots &  & a_n^{(m)} \eit & 0 & \cdots & 0 & a_n^{(-m)}& \cdots & a_n^{(-1)} &a_n^{(0)}
\end{pmatrix}
\end{equation}

Now we define a matrix which will be useful in what follows.
  
\begin{definition}\label{def:Cmu}
Fix $s \in {\mathbb N}$, $s\geq 2$, such that $s(n+1) \geq 2m+1$ and let $\mu:=s(n+1)$. We then define a $\mu \times \mu$ matrix $C_{\mu}$ as follows. For each $j,k=0, 1, \dots, \mu-1$, the $(j,k)$ entry of $C_\mu$ is
\[
\sum_{u \in \Z} \langle T e_{k+u \mu}, e_{j} \rangle.
\]
\end{definition}

Observe that all terms in the above sum are zero, except at most one. Indeed, the above sum is formed by adding some of the terms in the $j$-th row of $T$; namely, the entries $(j,k+u\mu)$ for $u \in \Z$. But assume there are two entries in the $j$-th row inside the band, the entries $(j,k+u_1\mu)$ and $(j,k+u_2\mu)$. Then, we must have
\[
-m \leq j-(k+u_1 \mu) \leq m \quad \text{ and } \quad -m \leq j-(k+u_2 \mu) \leq m
\]
but this implies
\[
- 2m \leq (u_2-u_1) \mu \leq 2m,
\]
which, since $\mu \geq 2m+1$ this is impossible unless $u_1=u_2$.

Notice that if the sequences $a^{(r)}$, for $r=-m, \dots, 0, \dots, m$, are regarded as $\mu$-periodic, instead of $n+1$ periodic, then $C_\mu$ is just the symbol $\Phi$ evaluated at $0$, and hence it has the same form as in expression \eqref{eq:matrix} (but it is of size $\mu \times \mu$ instead of $(n+1)\times (n+1)$). Hence if we erase the last $m$ columns and $m$ rows of $C_\mu$, we obtain $T_{\mu-m}$ a compression of the Toeplitz operator $T_+$. Therefore, we have
\begin{equation}\label{eq:TsubC}
W(T_{s(n+1) -m}) \subseteq W(C_{\mu})
\end{equation}
under the given condition $s(n+1) \geq 2m+1$.  

As in \cite{IM}, for  $0 \leq p \leq n$ and $0 \leq q < s$ we define the unit vectors in $\C^{\mu}$

\begin{align*}
    f_{p,q}
    &= \frac{1}{\sqrt{s}} \sum_{u=0}^{s-1} \rho^{u q} e_{p+u(n+1)}\\
    &=\frac{1}{\sqrt{s}} (\underbrace{0, \dots, 0}_{p}, 1, \underbrace{0, \dots, 0}_{n}, \rho^q, \underbrace{0, \dots, 0}_{n},
 \rho^{2q}, \underbrace{0, \dots, 0}_{n}, \dots,
    \rho^{(s-1)q}, \underbrace{0, \dots, 0}_{n-p})^\T,
\end{align*}
where $\rho=\exp(i 2 \pi/s)$, and we define the $\mu \times \mu$ unitary matrix  $U$ by
\[
U=\begin{pmatrix} 
  \vrule & \vrule & & \vrule &\vrule & \vrule & & \vrule & & &\vrule & \vrule &  &\vrule  \\
  f_{0,0} & f_{1,0} & \dots & f_{n,0} & f_{0,1} & f_{1,1} &\dots & f_{n,1} & \dots & \dots & f_{0,s-1} & f_{1,s-1} &\dots & f_{n,s-1} \\
    \vrule & \vrule & & \vrule &\vrule & \vrule & & \vrule & & & \vrule & \vrule & & \vrule
  \end{pmatrix}.
  \]

The matrix $U$ is useful in the next proposition.

\begin{proposition}\label{prop:Cmu}
Let $n$, $s$ and $\mu$ be as above. Then $C_\mu$ is unitarily equivalent to a block diagonal matrix. More precisely,   
\[
  U^{*}C_{\mu} U=\Phi(0) \oplus \Phi(2\pi/s) \oplus \Phi(4\pi/s) \oplus \cdots \oplus \Phi(2(s-1)\pi/s).
  \]
\end{proposition}
\begin{proof}
Let $p_1, p_2, q_1, q_2$ be integers such that $0 \leq q_1, q_2 \leq s-1$ and $0\leq p_1, p_2 \leq n$. Observe that
\begin{align*}
   \langle C_{\mu} f_{p_1, q_1}, f_{p_2, q_2} \rangle 
   & =\frac{1}{s} \sum_{u=0}^{s-1} \sum_{v=0}^{s-1} \langle C_{\mu} \rho^{u q_1} e_{p_1+u (n+1)}, \rho^{v q_2} e_{p_2+v (n+1)} \rangle \\
   &=\frac{1}{s}\sum_{u=0}^{s-1} \sum_{v=0}^{s-1}  \rho^{u q_1} \rho^{-v q_2}\langle C_{\mu} e_{p_1+u (n+1)}, e_{p_2+v (n+1)} \rangle \\ 
   &=\frac{1}{s}\sum_{u=0}^{s-1} \sum_{v=0}^{s-1}  \rho^{u q_1} \rho^{-v q_2}  \sum_{w \in {\mathbb Z}}  \langle 
      T e_{p_1 +u (n+1) +w \mu}, e_{p_2 +v (n+1)} \rangle \\
&=\frac{1}{s}\sum_{u=0}^{s-1} \sum_{v=0}^{s-1} \rho^{u q_1} \rho^{-v q_2} \sum_{w \in {\mathbb Z}}  \langle
      T e_{p_1 +(u-v) (n+1) +w \mu}, e_{p_2} \rangle,   
\end{align*}
where the last equality follows from the $(n+1)$-periodicity of $T$.

Making the change of variables $u=v+r$, we obtain
\begin{align*}
   \langle C_{\mu} f_{p_1, q_1}, f_{p_2, q_2} \rangle 
   & =\frac{1}{s}\sum_{v=0}^{s-1} \sum_{r=-v}^{-v+s-1}  \rho^{(v+r)q_1} \rho^{-v q_2}  \sum_{w \in {\mathbb Z}}  \langle
      T e_{p_1 + r (n+1) +w \mu}, e_{p_2} \rangle.
    \end{align*}
    
For each $v=0, 1, 2, \dots, s-1$, we consider the inner sum of the above expression and rewrite it as
\begin{align*}
    \sum_{r=-v}^{-v+s-1}  & \rho^{(v+r)q_1} \rho^{-v q_2}  \sum_{w \in {\mathbb Z}}   \langle
      T e_{p_1 + r (n+1) +w \mu}, e_{p_2} \rangle \\
      &= \sum_{r=-v}^{-1}  \rho^{(v+r)q_1} \rho^{-v q_2}  \sum_{w \in {\mathbb Z}}  \langle
      T e_{p_1 + r (n+1) +w \mu}, e_{p_2} \rangle + \sum_{r=0}^{-v+s-1}  \rho^{(v+r)q_1} \rho^{-v q_2}  \sum_{w \in {\mathbb Z}}  \langle
      T e_{p_1 + r (n+1) +w \mu}, e_{p_2} \rangle.
\end{align*}    
Making another change of variables in the first summand above, we obtain
\begin{align*}
    \sum_{r=-v}^{-1} & \rho^{(v+r)q_1} \rho^{-v q_2}  \sum_{w \in {\mathbb Z}}  \langle
      T e_{p_1 + r (n+1) +w \mu}, e_{p_2} \rangle + \sum_{r=0}^{-v+s-1}  \rho^{(v+r)q_1} \rho^{-v q_2}  \sum_{w \in {\mathbb Z}}  \langle
      T e_{p_1 + r (n+1) +w \mu}, e_{p_2} \rangle \\
      &= \sum_{r=s-v}^{s-1}  \rho^{(v+r-s)q_1} \rho^{-v q_2}  \sum_{w \in {\mathbb Z}}  \langle
      T e_{p_1 + (r-s) (n+1) +w \mu}, e_{p_2} \rangle + \sum_{r=0}^{-v+s-1}  \rho^{(v+r)q_1} \rho^{-v q_2}  \sum_{w \in {\mathbb Z}}  \langle
      T e_{p_1 + r (n+1) +w \mu}, e_{p_2} \rangle \\
&= \sum_{r=s-v}^{s-1}  \rho^{(v+r)q_1} \rho^{-v q_2}  \sum_{w \in {\mathbb Z}}  \langle
      T e_{p_1 + r(n+1) +(w-1) \mu}, e_{p_2} \rangle + \sum_{r=0}^{-v+s-1}  \rho^{(v+r)q_1} \rho^{-v q_2}  \sum_{w \in {\mathbb Z}}  \langle
      T e_{p_1 + r (n+1) +w \mu}, e_{p_2} \rangle \\
      &= \sum_{r=s-v}^{s-1}  \rho^{(v+r)q_1} \rho^{-v q_2}  \sum_{w \in {\mathbb Z}}  \langle
      T e_{p_1 + r(n+1) +w \mu}, e_{p_2} \rangle + \sum_{r=0}^{-v+s-1}  \rho^{(v+r)q_1} \rho^{-v q_2}  \sum_{w \in {\mathbb Z}}  \langle
      T e_{p_1 + r (n+1) +w \mu}, e_{p_2} \rangle \\
      &= \sum_{r=0}^{s-1}  \rho^{(v+r)q_1} \rho^{-v q_2}  \sum_{w \in {\mathbb Z}}  \langle
      T e_{p_1 + r(n+1) +(w+1) \mu}, e_{p_2} \rangle.
\end{align*}
    
Hence
\begin{equation}\label{eq:Cmfpq}
\begin{split}
   \langle C_{\mu} f_{p_1, q_1}, f_{p_2, q_2} \rangle 
   & = \frac{1}{s} \sum_{v=0}^{s-1} \sum_{r=0}^{s-1}  \rho^{(v+r)q_1} \rho^{-v q_2}  \sum_{w \in {\mathbb Z}}  \langle
      T e_{p_1 + r(n+1) +(w+1) \mu}, e_{p_2} \rangle \\
      &= \frac{1}{s}\sum_{w \in {\mathbb Z}}
       \sum_{r=0}^{s-1} 
       \sum_{v=0}^{s-1} \rho^{(v+r)q_1} \rho^{-v q_2}  \langle
      T e_{p_1 + r(n+1) +(w+1) \mu}, e_{p_2} \rangle  \\
      &=\frac{1}{s}\sum_{w \in {\mathbb Z}}
       \sum_{r=0}^{s-1} \rho^{r q_1}  \langle
      T e_{p_1 + r(n+1) +(w+1) \mu}, e_{p_2} \rangle 
       \sum_{v=0}^{s-1} \rho^{v(q_1-q_2)}.
       \end{split}
   \end{equation}

Now, assume $q_1 \neq q_2$. Since this implies that  $\sum_{v=0}^{s-1} \rho^{v(q_1-q_2)}=0$, equation \eqref{eq:Cmfpq} gives
\[
\langle C_{\mu} f_{p_1, q_1}, f_{p_2, q_2} \rangle =0.
\]

If $q_1=q_2$, then $\sum_{v=0}^{s-1} \rho^{v(q_1-q_2)}=s$ and hence equation \eqref{eq:Cmfpq} gives
\begin{align*}
\langle C_{\mu} f_{p_1, q_1}, f_{p_2, q_2} \rangle 
&=\sum_{w \in {\mathbb Z}}
       \sum_{r=0}^{s-1} \rho^{r q_1}  \langle
      T e_{p_1 + r(n+1) +(w+1) \mu}, e_{p_2} \rangle \\
&=\sum_{w \in {\mathbb Z}}
       \sum_{r=0}^{s-1} \rho^{r q_1}  \langle
      T e_{p_1 + (r+(w+1)s)(n+1)}, e_{p_2} \rangle \\
&=\sum_{w \in {\mathbb Z}}
       \sum_{r=0}^{s-1} \rho^{(r+(w+1)s)q_1}  \langle
      T e_{p_1 + (r+(w+1)s)(n+1)}, e_{p_2} \rangle,
\end{align*}
since $\rho^s=1$.

Observe that when $r$ ranges from $0$ to $s-1$ and $w$ ranges in $\Z$, the number $r+(w+1)s$ ranges over all the integers, and hence the previous expression can be written as
\[
\langle C_{\mu} f_{p_1, q_1}, f_{p_2, q_2} \rangle = \sum_{u \in {\mathbb Z}} \rho^{u q_1}  \langle
      T e_{p_1 + u (n+1)}, e_{p_2} \rangle,
\]
which equals the $(p_2,p_1)$ entry of the matrix $\Phi\left(\frac{2\pi q_1}{s}\right)$.

The above computations show that the matrix of $C_\mu$ with respect to the orthogonal basis $\{ f_{p,q} \, : \, 0 \leq p \leq n, \ 0  \leq q \leq s-1\}$ is block-diagonal and each of the $s$ blocks on the diagonal is of the form $\Phi\left(\frac{2\pi q}{s}\right)$, for $q=0, 1, \dots, s-1$. This proves the desired result.
\end{proof}

The proof of the next proposition follows the same reasoning as that of Corollary~2.4 in \cite{IM}. We include the details here for completeness.
 
 \begin{proposition}
Let $T_+$ be a $(n+1)$-periodic and $(2m+1)$-banded Toeplitz operator on $\ell^2(\N_0)$ and let $\Phi$ be its symbol. Then
\[
\ovl{W(T_+)}
   \subseteq \ovl{\conv{  \bigcup_{\theta\in[0,2\pi)}W(\Phi(\theta))}}.
   \]
 \end{proposition}
 \begin{proof}
 For every positive integer $s$ satisfying $(n+1)s >m$, we define the matrix $C_{(n+1)s}$ as in Definition~\ref{def:Cmu}. By Proposition~\ref{prop:Cmu}, we get 
\[
W(C_{(n+1)s})=\conv{ \bigcup_{k=0}^{s-1} W\left(\Phi\left(\tfrac{2k\pi}{s}\right)\right)}.
\]
By removing the last $m$ rows and the last $m$ columns in $C_{(n+1)s}$, we obtain  the matrix $T_{(n+1)s-m}$ and therefore, by the inclusion \eqref{eq:TsubC}, we have
 \[
 W(T_{(n+1)s -m}) \subseteq 
 \conv{ \bigcup_{k=0}^{s-1} W\left(\Phi\left(\tfrac{2k\pi}{s}\right)\right)}.
\]
Now, clearly
\[
\bigcup_{k=0}^{s-1} W\left(\Phi\left(\tfrac{2k\pi}{s}\right)\right) \subseteq 
\bigcup_{\theta\in[0,2\pi)}W(\Phi(\theta))
 \]
and thus it follows that 
\[W(T_{(n+1)s -m }) \subseteq 
\conv{  \bigcup_{\theta\in[0,2\pi)}W(\Phi(\theta))}.
\]
 Hence, since
\[
 W(T_1) \subseteq W(T_2) \subseteq W(T_3) \subseteq \dots,
 \]
 we obtain  
 \[
 \bigcup_{k=1}^{\infty} W(T_k) \subseteq 
 \conv{  \bigcup_{\theta\in[0,2\pi)}W(\Phi(\theta))}.
\]
Applying now Proposition 2.3 in \cite{IM}, we obtain the desired result. 
\end{proof}

For the next theorem, we first establish the following lemma.

\begin{lemma}\label{le:eigenvalues}
Let $n, s \in {\mathbb N}$, $s>1$, and let $\mu=s(n+1)$. For each $r=0, 1, 2, \dots, s-1$, if $\lambda$ is an
 eigenvalue of $\Phi(\tfrac{2 r \pi}{s})$ with eigenvector $\vec{v}=(v_0,v_1,\ldots,v_n) \in {\mathbb C}^{n+1}$, then $\lambda$ is an eigenvalue of the matrix $C_{\mu}$ with eigenvector
\[
\Vr=(v_0, \dots, v_n, v_0 \rho^r, \dots, v_n \rho^r, \dots, v_0 \rho^{(s-1)r}, \dots, v_n \rho^{(s-1)r}) \in \C^{\mu},
\]
where $\rho=\exp(i 2\pi/s)$. Conversely, if $\lambda$ is an eigenvalue of $C_{\mu}$, there is  $r=0, 1, 2, \dots, s-1$, such that $\lambda$ is an eigenvalue of $\Phi(\tfrac{2 r \pi}{s})$ with eigenvector $\vec{v}$ and the corresponding $\Vr$ is an eigenvector for $C_\mu$ with eigenvalue $\lambda$.
 \end{lemma}
\begin{proof}
A straightforward computation shows that $\Vr$ is indeed an eigenvector for the eigenvalue $\lambda$ of 
$C_{\mu}$ if $\vec{v}$ is eigenvector for the eigenvalue $\lambda$ of  $\Phi(\tfrac{2 r \pi}{s})$. 
    
For the last assertion, as a consequence of Proposition~\ref{prop:Cmu}, we have that
  \[
    \sigma(C_{\mu})=\bigcup_{r=0}^{s-1}\sigma\left(\Phi\left(\tfrac{2 r \pi}{s}\right)\right).
    \]
    Therefore, if $\lambda \in \sigma(C_{\mu})$ is given, there is $r$ such that $\lambda\in\sigma\left(\Phi(\tfrac{2r \pi}{s})\right)$. Let $\vec{v}$ be an eigenvector for $\Phi(\tfrac{2 r \pi}{s})$ corresponding to the eigenvalue $\lambda$. Then the corresponding vector $\Vr$ does the job.  
 \end{proof}

The following generalizes Theorem 2.6 in \cite{IM}.

\begin{theorem}\label{th:selfadjoint}
Let $n, m \in {\mathbb N}$ and let $T_+$ be an $(n+1)$-periodic and $(2m+1)$-banded selfadjoint Toeplitz operator on $\ell^2(\N_0)$ and let $\Phi$ be its symbol (which is a Hermitian matrix). Let $\lambda^-(\theta)$ and $\lambda^+(\theta)$ denote the 
smallest and largest eigenvalues of $\Phi(\theta)$ and set
\[
a=\min_{\theta\in[0,2\pi)}\lambda^-(\theta) \quad \text{ and } \quad b=\max_{\theta\in[0,2\pi)}\lambda^+(\theta).
\]
Then the closure of $W(T_+)$ is the closed interval
\[
\ovl{W(T_+)}=[a,b].
\]
\end{theorem}
\begin{proof}
The proof follows the same argument as that of Theorem 2.6 in \cite{IM}. We include the first part since some details are different and for the sake of completeness. 
We first show that $[a,b] \subseteq \overline{W(T_+)}$. Let $s\in {\mathbb N}$ such that $\mu:=s(n+1)\geq 2m+1$. Let $\lambda_s \in \sigma(C_{s(n+1)})$. 
 Lemma~\ref{le:eigenvalues} then implies that there exists $r=0, 1, 2, \dots, s-1$ such that  $\lambda_{s} \in \sigma\left(\Phi(\frac{2 r \pi}{s})\right)$. We can choose an eigenvector 
$\vec{v}=(v_0,v_1,\ldots,v_n)$ of $\Phi(\frac{2 r \pi}{s})$ of norm $1/\sqrt{s}$. In the manner of Lemma~\ref{le:eigenvalues}, the vector $\Vr$ is an eigenvector of norm $1$ for the matrix $C_{s(n+1)}$.

Recall that $T_{s(n+1)}$ denotes the compression of the operator $T_+$ to the subspace of $\ell^2(\N_0)$ spanned by the first $\mu=s(n+1)$ vectors in the canonical basis: $\{ e_0, e_1, e_2, \dots, e_{\mu-1}\}$. We then have
\begin{align*}
    \left\langle  T_{s(n+1)} \Vr, \Vr \right\rangle &=
    \left\langle  C_{s(n+1)} \Vr, \Vr \right\rangle 
    + \left\langle \left(T_{s(n+1)}-C_{s(n+1)}\right) \Vr, \Vr \right\rangle \\
    &= \lambda_{s} + \left\langle \left(T_{s(n+1)}-C_{s(n+1)}\right) \Vr, \Vr \right\rangle.
\end{align*}

Since the real number $\left\langle  T_{s(n+1)} \Vr, \Vr  \right\rangle$ is contained in $W(T_{s(n+1)}) \subseteq W(T_+)$, we see that
\begin{equation}\label{eq:lpTmCinW}
\lambda_{s} + \left\langle \left(T_{s(n+1)}-C_{s(n+1)}\right)\Vr, \Vr \right\rangle  \in W(T_+).
\end{equation}

Observe that the entries of $C_{s(n+1)}$ and $T_{s(n+1)}$ coincide everywhere except at the ``triangular corners'' on the top-right and down-left positions. Hence $b_{j,k}$, the $(j,k)$ entry of $C_{s(n+1)}-T_{s(n+1)}$, is zero if $|j-k|<s(n+1)-m$, and therefore at most $m(m+1)$ entries of the matrix $C_{s(n+1)}-T_{s(n+1)}$ are nonzero. If we set $\Vr=(w_0, w_1, w_2, \dots, w_{\mu-1})$, and we recall that the modulus of each coordinate of $\Vr$ is equal to the modulus of some coordinate of $\vec{v}$,  we have $|w_j|\leq \| \vec{v} \, \|$, and hence
\begin{align*}
    \left| \left\langle \left(C_{s(n+1)}-T_{s(n+1)}\right) \Vr, \Vr \right\rangle  \right|
    &\leq
    \sum_{j=0}^{\mu-1}    \sum_{k=0}^{\mu-1} |b_{j,k}| \ |w_j| \ |w_k| \\
    & \leq 
    \sum_{j=0}^{\mu-1}    \sum_{k=0}^{\mu-1} |b_{j,k}| \ \| \vec{v} \,\|^2 \\
    & \leq \max\left\{|b_{j,k}| \, : \, j, k =0, 1, 2, \dots, \mu -1\right\} \ m(m+1) \ \| \vec{v} \,\|^2.
\end{align*}

Since each nonzero $b_{j,k}$ is multiple, of modulus one, of an element of the finite set $\left\{a_j^{(k)}\right\}$, and since $\| \vec{v}\|^2=\frac{1}{s}$ it follows that 
\[
    \left| \left\langle \left(C_{s(n+1)}-T_{s(n+1)}\right) \Vr, \Vr \right\rangle  \right|
    \leq \max\left\{\left|a_j^{(k)}\right| \, : \, j=0, 1, \dots, n, \ k=-m, \dots, 0, \dots, m \right\} \ \frac{m(m+1)}{s}.
\]

Hence, the previous estimate ensures that if we can find a sequence $(\lambda_s)$ which converges as $s \to \infty$, then, by expression \eqref{eq:lpTmCinW}, we would have
\begin{equation} \label{lim_in_W}
\lim_{s \to \infty}\lambda_{s} \in  \overline{W(T_+)}.
\end{equation}

We now propose two suitable sequences of eigenvalues $(\lambda_{s})$. For each $s$, let $\lambda_s^-$ and
 $\lambda_s^+$ denote the smallest and largest eigenvalues of $C_{s(n+1)}$, respectively. 
We will prove that 
\begin{equation}\label{lim=min}
     \lim_{s \to \infty}\lambda_s^-=\min_{\theta\in[0,2\pi)}\lambda^-(\theta)
\end{equation}
 where $\lambda^-(\theta)$ is the smallest eigenvalue of $\Phi(\theta)$. It will follow then, by taking the limit as
 $s \to \infty$, that $a=\min_{\theta\in[0,2\pi)}
\lambda^-(\theta)\in \overline{W(T_+)}$.

By Proposition~\ref{prop:Cmu} and  since $\lambda_s^-$ is the smallest eigenvalue of $C_{s(n+1)}$, we have 
\[
\sigma(C_{s(n+1)})=\bigcup_{r=0}^{s-1}\sigma\left(\Phi\left(\frac{2 r \pi }{s}\right)\right).
\]
Bu then $\lambda_s^-$ is not only an eigenvalue of $\Phi(\frac{2\pi r}{s})$ for some $r$ but it is, in
 fact, the smallest among the eigenvalues of all symbols $\Phi(\frac{2 r \pi}{s})$ for $r=0, 1, \dots, s-1$; i.e.,
 $\lambda^-_s=\lambda^-\left(\frac{2 r \pi}{s}\right)$ for some $r=0,1,2, \dots, s-1$.

Let 
\[ 
\lambda^-(\theta^*)=\min_{\theta\in[0,2\pi)}\lambda^-(\theta)
\]
where $\theta^*$ is the point where the minimum is reached.
Therefore $\lambda^-(\theta^*)\leq \lambda^-\left(\frac{2 r \pi}{s}\right)=\lambda_s^-$. Using the continuity of $\lambda^-(\cdot)$, for all 
$\epsilon>0$ there is $\delta>0$ such that $|\theta-\theta^*|<\delta$ implies
 $\lambda^-(\theta^*)\leq\lambda^-(\theta)<\lambda^-(\theta^*)+\epsilon$. Also, there exists $N \in \N$ such that for $s\geq N$ we may choose $0\leq j\leq s-1$ such that $\left|\frac{2 j \pi}{s} - \theta^*\right|<\delta$. Thus for all $s\geq N$, 
\[
  \lambda^-(\theta^*)\leq\lambda_s^-=\lambda^-\left(\frac{2 r \pi}{s}\right)\leq\lambda^-\left(\frac{2 j \pi}{s}\right)<\lambda^-(\theta^*)+\epsilon.
\]
This proves Equation~\eqref{lim=min}. Analogously, one can show that
\[
\lim_{s\to\infty}\lambda_s^+=\max_{\theta\in[0,2\pi)}\lambda^+(\theta),
\]
where $\lambda_s^+$ is the largest eigenvalue of $C_{s(n+1)}$ and $\lambda^+(\theta)$ is the largest eigenvalue of $\Phi(\theta)$. 

Hence, Expression~\eqref{lim_in_W} proves that
\[
  a= \min_{\theta\in[0,2\pi)}\lambda^-(\theta)=\lim_{s\to\infty}\lambda_s^-\in\overline{W(T_+)}
\]
and 
\[
  b= \max_{\theta\in[0,2\pi)}\lambda^+(\theta)=\lim_{s\to\infty}\lambda_s^+\in\overline{W(T_+)}.
\]
Convexity then shows that $[a,b]\subseteq \overline{W(T_+)}$. 

The proof of the inclusion $\overline{W(T_+)} \subseteq [a,b]$ is now exactly the same as that of Theorem 2.6 in \cite{IM} and we omit the proof.
\end{proof}

%By using the comparison principle mentioned in Lemma 2.7 in \cite{IM}, as

Acting as in the proof of Theorem 2.8 in \cite{IM}, 
the above Theorem~\ref{th:selfadjoint} implies the following result.

\begin{theorem}\label{th:numranT+}
Let $n, m \in \N$ and let $T_+$ be an $(n+1)$-periodic and $(2m+1)$-banded Toeplitz operator acting on $\ell^2 (\N_0)$. Let $\Phi(\theta)$ be the symbol of $T_+$. Then 
\[
   \overline{W(T_+)} =\ovl{ \conv{ \bigcup_{\theta \in [0, 2\pi)} W(\Phi(\theta))}}.
   \]
\end{theorem}

Also, imitating the proof of Proposition 1.1 in \cite{IM} we obtain that $\ovl{W(T)}=\ovl{W(T_+)}$ and hence we have the following result.

\begin{corollary}
Let $n, m \in \N$ and let $T$ be an $(n+1)$-periodic and $(2m+1)$-banded Toeplitz operator acting on $\ell^2 (\Z)$. Let $\Phi(\theta)$ be the symbol of $T$. Then 
\[
   \overline{W(T)} =\ovl{ \conv{ \bigcup_{\theta \in [0, 2\pi)} W(\Phi(\theta))}}.
   \]
\end{corollary}

\section{The boundary of the numerical range of a periodic banded Toeplitz operator}

In this section we shall consider the boundary of the numerical range $W(T_+)$ of the $(n+1)$-periodic and $(2m+1)$-banded Toeplitz operator $T_+$. By Theorem~\ref{th:numranT+}, every boundary point of $W(T_+)$ lies on the boundary of the set 
\[
{\mathcal O}:=\bigcup_{\theta \in [0, 2 \pi)} W(\Phi(\theta))
\]
or on one of the support lines of $\mathcal O$. Each symbol matrix $\Phi(\theta)$ is an $(n+1) \times (n+1)$  complex matrix whose entries are trigonometric polynomials in  $\theta$. So we can express $\Phi(\theta)$ as a finite sum of the form
    \[
    \Phi(\theta)= \sum_{r=-L}^L A(r) e^{i r \theta},
    \]
    for some $L \in \N$ and
    where each $(n+1)\times(n+1)$ matrix $A(r)$ does not depend on $\theta$. Hence, for each $\xi=(\xi_0, \xi_1, \dots, \xi_n) \in \C^{n+1}$, we have
    \[
    \langle \Phi(\theta) \xi, \xi \rangle =\sum_{r=-L}^{L} \langle A(r) e^{i r \theta} \xi, \xi \rangle = \sum_{r=-L}^{L} e^{i r \theta} \sum_{j=0}^n \sum_{k=0}^n a_{j,k}(r) \xi_j \ovl{\xi_k}, 
    \]
    where $a_{j,k}(r)$ is the $(j,k)$ entry of the matrix $A(r)$. Hence
\[
{\mathcal O}=\left\{ \sum_{r=-L}^L e^{i r \theta} \sum_{j=0}^{n} \sum_{j=0}^{n}  a_{j,k}(k) \xi_{j} \overline{\xi_k} :
    \xi=(\xi_0, \xi_1, \dots, \xi_{n}) \in {\mathbb C}^{n+1}, \| \xi \| =1,  0 \leq \theta < 2\pi \right\}.
    \]
    
    Under the identification of ${\mathbb C}$ with ${\mathbb R}^2$, the set
\[
\{ \xi \in {\mathbb C}^{n+1} \ : \ \| \xi \|=1 \} \times S^1 
\]
can be written as
\[ 
\{ (x_0, y_0, x_1, y_1, \dots, x_n y_n) \in \R^{2n+2} \ : \ x_0^2 + y_0^2 + x_1^2+y_1^2 + \dots x_n^2 + y_n^2 =1 \} \times \{(x, y) \in {\mathbb R}^2: x^2 +y^2 =1 \},
\]
which is an algebraic set (see, for example, \cite{Bene-Risl}). But then, ${\mathcal O}$ is the image of a real algebraic set under a polynomial map from $\R^{2n+4}$ to $\R^2$. Then, by the Tarski-Seidenberg Theorem (cf. \cite{Bene-Risl}, page 60, Theorem 2.3.4), ${\mathcal O}$ is a semi-algebraic set. Hence, the boundary of $\mathcal O$ lies on some real algebraic curve; that is, there is a non-zero real polynomial $f(x, y)$ for which every boundary point $(x, y)$ of ${\mathcal O}$ satisfies
 \[
 f(x, y)=0.
 \]
Since the set $\{ \xi \in \C^{n+1} \, : \, \| \xi \|=1 \} \times S^1$ is  compact and connected, its image ${\mathcal O}$ under 
a continuous map is also compact and connected. Since the number of connected components of a semialgebraic set is finite \cite[Theorem 2.4.5]{BoCoRo}, this implies that the boundary of $\mathcal O$ has a finite number of connected components. 

Consider the set of polynomials $f(x,y) \in \R[x,y]$ such that
\[
\partial {\mathcal O} \subseteq \{ (x,y) \in \R^2 \, : \, f(x,y)= 0 \}.
\]
By the comment above, this set is nonempty. We choose $f_0(x,y) \in \R[x,y]$, of smallest degree, satisfying 
\[
\partial {\mathcal O} \subseteq \{(x, y) \in {\mathbb R}^2:f_0(x, y)=0 \}.
\]
Clearly, since $\R[x,y]$ is a unique factorization domain, we can write $f_0(x,y)$ as
\[
f_0=f_1^{m_1} f_2^{m_2} \cdots f_k^{m_k}
\]
where $f_1, f_2, \ldots, f_k$ are the mutually inequivalent irreducible factors of $f_0$ and the numbers $m_j$ are positive integers. Since $f_0(x,y)$ is of minimum degree, clearly we must have that $m_j=1$ for all $j$.

The curve $\{(x, y) \in \R^2: f_0(x,y)=0 \}$ may have multitangents. Such a multitangent corresponds to some singular point of the dual curve of the multiplicity free curve $f_0(x, y)=0$; therefore, the number of such lines is at most finite. We denote the union of such multitangents as $g_0(x, y)=0$ using some 
polynomial $g_0(x, y) \in {\mathbb R}[x,y]$. By letting $g(x, y)=f_0(x, y) g_0(x,y)$, we have the following inclusion
\[
\partial \conv{{\mathcal O}} \subseteq \{(x, y): g(x, y) =0 \}.
\]

We have obtained the following theorem.
\begin{theorem}
Let $T_+$ be an $(n+1)$-periodic and $(2m+1)$-banded Toeplitz operator acting on $\ell^2({\mathbb N_0})$. Then there exists a nonzero real binary polynomial $g(x, y)$ such that if $ (x, y)
\in \partial W(T_+)$, then $g(x, y)=0$.
\end{theorem}
    
 \section{A Toeplitz operator for which the closure of the numerical range is not the numerical range of a matrix} 

We examine the following example of a $2$-periodic and $5$-banded Toeplitz operator $T_+$. We set the sequences $a^{(0)}$, $a^{(-1)}$ and $a^{(-2)}$ to be the constant sequences of zeroes, we set the sequence $a^{(1)}=\left(a_j^{(1)}\right)_{j\in \Z}$ with $a_j^{(1)}=-1$ if $j$ is an even integer and $a_j^{(1)}=2$ if $j$ is an odd integer; and we set the sequence $a^{(2)}$ to be the constant sequence of ones. That is, the matrix of $T_+$ has the upper triangular form
\begin{equation}\label{ex:counter}
\begin{pmatrix}
0 & -1 & 1 &  0 & \cdots \\
0 & 0  & 2 &  1 & 0 & \cdots  \\
0 & 0  & 0 & -1 & 1 & 0 & \cdots \\
0 & 0 & 0  & 0 & 2 & 1 & 0 & \cdots \\
\vdots & 0 & 0 & 0  & 0 & -1 & 1 & 0 & \cdots \\
\vdots & \ddots &\ddots &\ddots &\ddots 
\end{pmatrix}.
\end{equation}
By Definition~\ref{def:symbol}, the symbol of $T_+$ is the $2\times 2$ matrix function
\[
\Phi(\theta)=\begin{pmatrix}
e^{i \theta} & -1 \\
2 e^{i\theta} & e^{i\theta}
\end{pmatrix}.
\]
The numerical range of $\Phi(\theta)$ is the convex set bounded by the ellipse with focii the points 
\[
e^{i\theta}+ i \sqrt{2}  e^{i \theta/2} \quad \text{ and } \quad e^{i\theta}- i \sqrt{2} e^{i \theta/2},
\]
major axis of length $3$, and minor axis of length $1$.

By Theorem~\ref{th:numranT+}, the closure of $W(T)$ is given as the closure of the convex hull of the union of the above ellipses.

Each ellipse can be parametrized as
\begin{align*}
X(t)&=\cos(\theta) + \frac12 \cos(\theta/2) \cos(t) - \frac32 \sin(\theta/2) \sin(t),\\
Y(t)&=\sin(\theta) + \frac12 \sin(\theta/2) \cos(t) + \frac32 \cos(\theta/2) \sin(t),
\end{align*}
where $t\in [0, 2\pi)$. The Cartesian equation of each ellipse is
\[
(20 +16 \cos \theta)(X -\cos \theta)^2 +(20-16 \cos \theta)(Y-\sin \theta)^2 +32\sin \theta (X -\cos \theta)(Y -\sin \theta)=9.
\]
or equivalently,
\[
(16 X^2-16 Y^2 - 40 X +16 )\cos \theta + (32 X Y -40 Y)  \sin \theta + (20 X^2 + 20 Y^2 - 32 X + 11)=0.
\]

We now consider the function
\begin{align*}
H&(X,Y;\theta) \\
=& (16 X^2-16 Y^2 - 40 X +16 )\cos \theta + (32 X Y -40 Y)  \sin \theta + (20 X^2 + 20 Y^2 - 32 X + 11).
\end{align*}

Observe that, for each $\theta \in [0, 2 \pi)$, the boundary of the numerical range of $\Phi(\theta)$ is given by the ellipse $H(X,Y;\theta)=0$. To find the envelope of this family of ellipses, we write
\[
H(X,Y;\theta) = \alpha \cos \theta + \beta \sin \theta + \gamma
\]
where
\[
\alpha = 16 X^2-16 Y^2 - 40 X +16, \quad \beta=32 X Y -40 Y, \quad \text{ and } \quad \gamma=20 X^2 + 20 Y^2 - 32 X + 11.
\]

As is well-known \cite[p. 76]{BG}, the envelope consists of those points $(X,Y)$ such that there exists $\theta \in [0,2 \pi)$ such that
\[
H(X,Y;\theta)=0 \quad \text{ and } \quad \frac{\partial}{\partial \theta} H(X,Y;\theta)=0.
\]
In other words, to find the points $(X,Y)$ in the envelope, we need to check if there exists $\theta \in [0, 2 \pi)$ such that the system
\[
\alpha \cos \theta + \beta \sin \theta + \gamma =0, \quad \text{ and } \quad 
\beta \cos \theta - \alpha \sin \theta = 0,
\]
has a solution. A sufficient and necessary condition for this system to have  a solution is that $\alpha^2 + \beta^2 = \gamma^2$; that is
\[
(16 X^2-16 Y^2 - 40 X +16)^2+(32 X Y -40 Y)^2 =(20 X^2 + 20 Y^2 - 32 X + 11)^2.
\]
Equivalently,
\begin{equation}\label{eq:quartic}
    16X^4 +32X^2 Y^2 +16Y^4 -72 X^2 -72  Y^2 +64 X-15=0.
\end{equation}

Solving for $Y$ and using the second derivative test for convexity, it is straightforward to check that the points satisfying the above equation are the boundary of a convex set, except for the isolated point $(\frac12,0)$ which lies in the interior of said set. Hence, the union of the ellipses $H(X,Y;\theta)=0$ and their interiors is a convex set. Also it is clear that the union of the family of the ellipses $H(X,Y;\theta)=0$ and their interiors is a closed set. Therefore, by Theorem~\ref{th:numranT+}, we conclude that a point $X + i Y$ is in the boundary of the numerical range of the operator $T_+$ if and only if it satisfies the equation
\[
16X^4 +32X^2 Y^2 +16Y^4 -72 X^2 -72  Y^2 +64 X-15=0,
\]
and $(X,Y)\neq (\frac{1}{2},0)$.

Hence, the boundary of $W(T_+)$ consists of the points described by the set
\[
\left\{ (X, Y) \in {\mathbb R}^2: L(1, X, Y)=0 \right\} \setminus \left\{ (\tfrac{1}{2},0) \right\}
\]
where 
\begin{equation}\label{eq:L}
L(U, X, Y)=16X^4 +32X^2 Y^2 +16Y^4 -72U^2 X^2 -72 U^2 Y^2 +64U^3 X-15 U^4,
\end{equation}
see Figure \ref{fig:numrange}.

\begin{figure}\label{fig:numrange}    \includegraphics[scale=0.4]{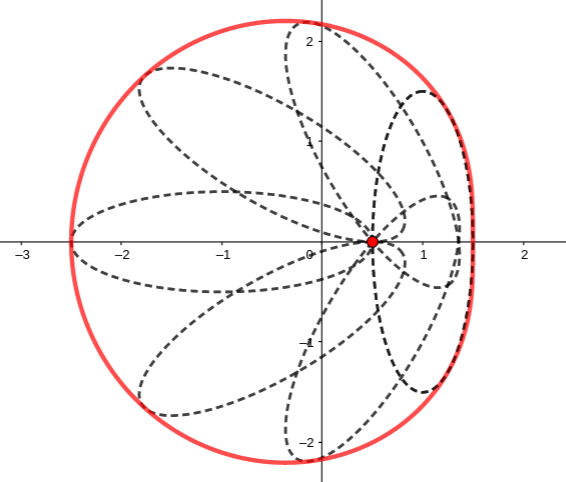}
\caption{The red solid curve represents the points such that $L(1,X,Y)=0$ and the black dotted curves represent the ellipses $H(X,Y;\theta)=0$ for six different values of $\theta$.}
\end{figure}

The complex projective curve 
\[
\left\{ [U; X; Y] \in {\mathbb CP}^2: L(U, X, Y)=0 \right\}
\]
has an ordinary double point at $(U,X,Y)=(2,1,0)$ and a pair of $(2,3)$-cusps at $(U,X,Y)=(0,1, \pm i)$. This curve has no other singular points.
These properties imply that the quartic polynomial $L(U,X,Y)$ is irreducible in the polynomial ring $\C[U,X,Y]$ (see \cite{ChNa2012}). 

Also, we consider the dual curve $\ell(t, x, y)=0$ of the curve $L(U, X, Y)=0$ under the pairing 
\[
t U +x X +y Y =0.
\]
The curve $\ell(t, x, y)=0$ is also an irreducible quartic algebraic curve and a computation shows that $\ell$ is given by
\begin{equation}\label{eq:ell}
\ell(t, x, y)=16 t^4 +32 t^3 x -72 t^2 (x^2 +y^2) -216 t (x^3 +x y^2) -135 x^4 -162 x^2 y^2 -27y^4.
\end{equation}

The irreducibility of $L$ and $\ell$ will be useful when answering the main question of this section: Is  there some $k \times k$ complex matrix $B$ for which 
    \[
    \ovl{W(T_+)} =W(B)?
    \]
The following theorem shows that the answer is negative.

\begin{theorem}\label{th:TnotB}
Let $a^{(0)}$, $a^{(-1)}$ and $a^{(-2)}$ be the constant sequences of zeroes, let  $a^{(1)}=\left(a_j^{(1)}\right)_{j \in \Z}$ be the sequence defined by $a_j^{(1)}=-1$ if $j$ is an even integer and $a_j^{(1)}=2$ if $j$ is an odd integer; and let  $a^{(2)}$ be the constant sequence of ones. Suppose that $T_+$ is the biperiodic and $5$-banded Toeplitz operator acting on $\ell^2(\N_0)$ defined by these sequences, which has a matrix representation given in  \eqref{ex:counter}. For every $k\in \N$, there is no $k \times k$ matrix $B$ such that 
\[
    \ovl{W(T_+)} =W(B).
\]
\end{theorem}

Before proving this theorem, we introduce a notion which was originally formulated for the principal symbol of a partial differential operator. 

\begin{definition}
Suppose that $F(t, x, y)$ is a real homogeneous polynomial of degree $k \geq 1$. The form $F$ is said to be {\em hyperbolic} with respect to $(1,0,0)$ if for every nonzero $(x_0,y_0) \in \R^2$, the equation
\[
 F(t, x_0, y_0)=0
 \]
 as a polynomial in the variable $t$ has $k$ real solutions counting multiplicities. 
\end{definition}

Observe that if $F_B$ is the Kippenhahn polynomial associated with an $k \times k$ matrix $B$; i.e.,
 \[
 F_{B}(t, x, y)={\rm det}(t I_m +x \Re(B) +y \Im(B)),
 \]
then the equation $F_{B}(t, x, y)=0$ is the characteristic equation of the Hermitian matrix $-x \Re(B) -y \Im(B)$ and hence, for each $(x,y) \in \R^2$, it has $k$ real roots counting  multiplicites. So the form $F_{B}$ is hyperbolic with respect to $(1,0,0)$  (cf. \cite{Kip}). 

\begin{proof}[Proof of Theorem~\ref{th:TnotB}]
By contradiction, we assume that for some $k \in \N$, there exists a $k \times k$ matrix $B$ satisfying  $W(B)=\ovl{W(T_+)}$. 

By Kippenhahn's theorem, the numerical range of a matrix $B$ equals the convex hull of the real affine part of the dual curve of $F_B(t,x,y)=0$. Let us denote by $f_B(U,X,Y)$ the dual curve of $F_B(t,x,y)$. As we have seen, a point $(X_0,Y_0)$ on the boundary of the numerical range of $T_+$ satisfies the equation $L(1,X_0,Y_0)=0$, where $L(U,X,Y)$ is given in \eqref{eq:L}. Then the assumption implies that $f_B(1,X_0,Y_0)=0$. By B\'ezout's Theorem, $L(U,X,Y)$ and $f_B(U,X,Y)$ must have a common component. Since $L(U,X,Y)$ is irreducible, we conclude that $L(U,X,Y)$ is a factor of $f_B(U,X,Y)$. It follows that its dual curve, $\ell(t,x,y)$, described in \eqref{eq:ell}, must be a factor of $F_B(t,x,y)$.

By the relation (3.12) of \cite[p. 130]{ABG}, the ternary form $\ell(t, x, y)$ as a factor of $F_{B}(t, x, y)$ has to be hyperbolic with respect to $(1,0,0)$, that is, the quartic equation $\ell(t, -\cos \theta, -\sin \theta)=0$ in $t$ has $4$ real solutions counting its multiplicities for any angle $0 \leq \theta < 2\pi$. But the equation for $\theta=\pi/2$ is expressed as 
\[
16 t^4 -72 t^2 -27= \left( 4 t^2 -(6\sqrt{3}+9) \right) \,  \left( 4 t^2 +(6 \sqrt{3}-9)\right)=0.
\]
The equation $4 t^2 -(6 \sqrt{3} +9)=0$ has two real solutions, but since $6 \sqrt{3}-9 > 0$, the equation $4 t^2 +(6 \sqrt{3}-9)$ has no real solutions. Hence, the quartic equation $16t^4-72 t^2 -27=0$ in $t$ does not have only real solutions. So the ternary form $\ell(t, x, y)$ is not hyperbolic with respect to $(1,0,0)$. This contradiction implies that our assumption on the existence of  a $k\times k$ matrix $B$ with $W(B)=\ovl{W(T_+)}$ is false. 
\end{proof}

\bibliographystyle{plain}

\bibliography{mybib}{}

\end{document}